\documentclass[11pt]{amsart}
\usepackage{amssymb,times,url}
\newtheorem{theorem}{Theorem}[section]
\newtheorem{lemma}[theorem]{Lemma}
\newtheorem{proposition}[theorem]{Proposition}
\newtheorem{corollary}[theorem]{Corollary}
\theoremstyle{remark}
\newtheorem{remark}[theorem]{Remark}

\def\gp#1{\langle \hspace*{.35mm} #1  \rangle}

\newcommand{\F}{\mathbb{F}}

\newcommand{\Z}{\mathbb{Z}}
\newcommand{\Q}{\mathbb{Q}}
\newcommand{\C}{\mathbb{C}}
\renewcommand{\P}{\mathbb{P}}
\newcommand{\GL}{\mathrm{GL}}
\newcommand{\rk}{\mathrm{rk}}
\newcommand{\hi}{\mathrm{h}}
\newcommand{\UT}{\mathrm{UT}}

\begin{document}

\title{Algorithms for linear groups of finite rank}

\author{A.~S. Detinko, D.~L. Flannery, and E.~A. O'Brien}




\begin{abstract}
Let $G$ be a finitely generated solvable-by-finite linear group.
We present an algorithm to compute the torsion-free rank of $G$
and a bound on the Pr\"{u}fer rank of $G$. This yields in turn an
algorithm to decide whether a finitely generated subgroup of $G$
has finite index. The algorithms are implemented in {\sc Magma}
for groups over algebraic number fields.
\end{abstract}

\maketitle

In \cite{Tits, Recog} we developed practical methods for computing
with linear groups over an infinite field $\F$. Those methods were
used to test whether a finitely generated subgroup of $\GL(n,\F)$
is solvable-by-finite (SF). We now proceed to the design of
further algorithms for finitely generated SF linear groups. Such a
group may not be finitely presentable (see \cite[4.22,
\mbox{p.}~66]{Wehrfritz}), so obviously cannot be studied using
approaches that 
 require a presentation; in contrast to, say,
polycyclic-by-finite (PF) groups. Extra restrictions are necessary
to make computing feasible. Groups of finite rank are suitable
candidates from this point of view, because they are well-behaved
algorithmically \cite[Section~9.3]{LenRob}.
They also have convenient structural features (see
\cite[Section~5.2]{LenRob} and Section~\ref{Prelim}).


In this paper we develop initial results to enable computing with
finitely generated linear groups of finite rank. Since such groups
are $\Q$-linear (Proposition~\ref{MainlyKopytov}), our primary
focus is the case that $\F$ is an algebraic number field. We first
test whether $G\leq \GL(n, \F$) has finite rank. If so, we compute
its torsion-free rank and an upper bound on its Pr\"{u}fer rank.
This furnishes an algorithm to decide whether a finitely generated
subgroup of $G$ has finite index. We determine various asymptotic
bounds of interest in their own right.
%
%
%
Algorithms for the structural investigation of $G$ are provided as
well: these construct a completely reducible part, and a finitely
generated subgroup with the same rank as the unipotent radical.
Our algorithms have been implemented in {\sc Magma} \cite{Magma}.
We emphasize that computations are performed with a given group in
its original representation, avoiding enlargement of matrices to
get an isomorphic copy over $\Q$.
%
%


Naturally, it is possible to take advantage of additional
properties of $G$ when they are known. If $G$ is 
polycyclic then one could obtain its torsion-free rank from a
consistent polycyclic presentation of $G$,
the latter found as in \cite{BettinaBjorn}. An even more tractable
class is nilpotent-by-finite groups (\mbox{cf.}
\cite[Section~7]{Dixon85}).
%
%
%

We summarize the layout of the paper. Section~\ref{Prelim} gives
background on linear groups of finite rank, including a reduction
to SF groups over a number field. Section~\ref{SFgroupsoverP} is
an extended treatment of such groups. In Section~\ref{Sect3} we
discuss ranks of finite index
subgroups; we are indebted to D.~J.~S.~Robinson 
for a vital theorem here. Section~\ref{Sect3} also shows how to
find the rank of a unipotent normal subgroup. In
Section~\ref{AlgSect} we present our algorithms and some
experimental results.

Unless stated otherwise, $\F$ is an (infinite) field. The rational
field is denoted as usual by $\Q$, and $\P$ is a number field with
ring of integers $\mathcal{O}_\P$.

\section{Preliminaries}\label{Prelim}

A general reference for this section is \cite[Chapter~5]{LenRob}.

\subsection{Pr\"{u}fer rank and torsion-free rank}

Recall that a group $G$ has finite Pr\"{u}fer rank $\rk(G)$ if
each finitely generated subgroup of $G$ can be generated by
$\rk(G)$ elements, and $\rk(G)$ is the least such integer.
\begin{theorem}\label{GroundZero}
Let $G\leq \GL(n,\F)$ have finite Pr\"{u}fer rank. Then $G$ is SF,
and if $\mathrm{char} \, \F >0$ then $G$ is abelian-by-finite
(AF).
\end{theorem}
\begin{proof}
See \cite[10.9, \mbox{p.}~141]{Wehrfritz}.
\end{proof}
\begin{corollary}\label{FGCredFRisAF}
Let $G$ be a finitely generated subgroup of $\GL(n,\F)$. If $G$ is
AF then it has finite Pr\"{u}fer rank; if $G$ is completely
reducible and has finite Pr\"{u}fer rank then it is AF.
\end{corollary}
\begin{proof}
If $G$ is AF then it has a normal finitely generated abelian
subgroup $A$ of finite index. Since $A$ and $G/A$ have finite
rank, so too does $G$. On the other hand, if $G$ is completely
reducible and has finite rank, then it is AF by
Theorem~\ref{GroundZero} and \cite[3.5 (ii),
\mbox{p.}~44]{Wehrfritz}.
\end{proof}
\begin{remark}
The converse of Theorem~\ref{GroundZero} is not true even when $G$
is finitely generated. However, see Proposition~\ref{FgSFoverP}.
\end{remark}

\begin{proposition}\label{MainlyKopytov}
If $G$ is a finitely generated subgroup of $\GL(n,\F)$ of finite
Pr\"{u}fer rank then $G$ is $\Q$-linear, i.e., isomorphic to a
subgroup of $\GL(d,\Q)$ for some $d$.
\end{proposition}
\begin{proof}
Suppose that $\mathrm{char} \, \F = 0$. By \cite[4.8,
\mbox{p.}~56]{Wehrfritz}, $G$ is (torsion-free)-by-finite, and by
Theorem~\ref{GroundZero}, $G$ is SF. Thus $G$ contains a
torsion-free solvable normal subgroup of finite index and finite
rank. The result now follows from \cite[Theorem 2]{Kopytov}.

Suppose that $\mathrm{char} \, \F > 0$. By
Theorem~\ref{GroundZero}, $G$ is PF.
It is well-known that a PF group is $\Z$-linear; see \cite[3.3.1,
\mbox{p.}~57]{LenRob}.
\end{proof}

Theorem~\ref{GroundZero} and Proposition~\ref{MainlyKopytov}
essentially reduce the investigation of finitely generated linear
groups of finite rank to the case of SF groups over $\Q$. In
Section~\ref{SFSubgpsGLnP} we show conversely that finitely
generated SF subgroups of $\GL(n,\P)$ always have finite rank.
Hence we restrict attention mainly to groups over number fields.

Now recall that a group $G$ has finite torsion-free rank if it has
a subnormal series of finite length whose factors are either
periodic or infinite cyclic. The number $\hi(G)$ of infinite
cyclic factors is the {\em Hirsch number}, or {\em torsion-free
rank}, of $G$.

\begin{lemma}\label{FinitePMeansFiniteH}
An SF group with finite Pr\"{u}fer rank has finite torsion-free
rank.
\end{lemma}
\begin{proof}
See \cite[\mbox{p.}~85]{LenRob}.
\end{proof}

\begin{lemma}\label{RkHFormulae}
Let $G$ be a group with normal subgroup $N$.
\begin{itemize}
\item[{\rm (i)}] If $G$ has finite Pr\"{u}fer rank then
$\rk(G)\leq\rk(N) + \rk(G/N)$. \item[{\rm (ii)}] If $G$ has finite
torsion-free rank then $\hi(G) = \hi(N) + \hi(G/N)$.
\end{itemize}
\end{lemma}


\subsection{Polyrational groups}\label{PolyratSect}

Let $U(G)$ be the unipotent radical of $G\leq \GL(n,\F)$; namely,
the largest unipotent normal subgroup of $G$. Note that $G/U(G)$
is isomorphic to a completely reducible subgroup of $\GL(n,\F)$.
If we exhibit $G$ in block triangular form with completely
reducible blocks, then $U(G)$ is the kernel of the projection of
$G$ onto its main diagonal. Denote the largest periodic normal
subgroup of $G$ by $\tau(G)$.
\begin{lemma}\label{TauGFinite}
Let $G$ be a finitely generated subgroup of $\GL(n,\F)$ of finite
Pr\"{u}fer rank. Then $\tau(G)$ is finite.
\end{lemma}
\begin{proof}
Theorem~\ref{GroundZero} and Proposition~\ref{MainlyKopytov} imply
that $G$ is SF and we may assume that $\mathrm{char} \, \F = 0$.
Then $\tau(G)$ is isomorphic to a subgroup of $\tau(G/U(G))$, and
$G/U(G)$ is finitely generated AF by Corollary~\ref{FGCredFRisAF}.
So we may further assume that $G$ has a normal abelian subgroup
$A$ of finite index. Since $A$ is finitely generated, $\tau(G)\cap
A \leq \tau(A)$ is finite. Thus $|\tau(G)| = |\tau(G)A:A|\cdot
|\tau(G)\cap A|$ is finite.
\end{proof}

A group is {\em polyrational} if it has a series of finite length
with each factor isomorphic to a subgroup of the additive group
$\Q^+$. So a polyrational group has finite torsion-free and
Pr\"{u}fer ranks. 
\begin{proposition}\label{RkeqHPolyrat}
If $G$ is polyrational then $\rk(G) = \hi(G)$.
\end{proposition}
\begin{proof}
See \cite[5.2.7, \mbox{p.}~93]{LenRob}.
\end{proof}
\begin{theorem}\label{FGFinitePrRkiffPolyratBF}
A finitely generated subgroup $G$ of  $\hspace*{.1mm} \GL(n,\F)$
has finite Pr\"{u}fer rank if and only if it is
polyrational-by-finite. In this case, $\hi(G) \leq \rk(G)$.
\end{theorem}
\begin{proof}
The first statement follows from Theorem~\ref{GroundZero},
Lemmas~\ref{FinitePMeansFiniteH} and \ref{TauGFinite}, and
\cite[5.2.5, \mbox{p.}~92]{LenRob}. For the second, let $N$ be a
normal polyrational finite index subgroup of $G$; then $\hi(G) =
\hi(N)= \allowbreak \rk(N) \leq \rk(G)$.
\end{proof}
Henceforth, the term `rank' without a qualifier means Pr\"{u}fer
or torsion-free rank.

\section{Solvable-by-finite groups over a number field}
\label{SFgroupsoverP}

We now focus on finitely generated SF subgroups of $\GL(n,\P)$.
Set $|\P:\Q|=m$. In this section we obtain more detailed
information about these groups that will be used in our
algorithms.


A finitely generated subgroup $G$ of $\GL(n,\F)$ is contained in
$\GL(n,R)$ where $R\subseteq \F$ is a finitely generated integral
domain. The quotient ring $R/\rho$ is a finite field for any
maximal ideal $\rho$ of $R$. We explain in \cite[Section~2]{Tits}
how to construct a congruence homomorphism $\varphi_\rho:
\allowbreak \GL(n,R) \rightarrow \allowbreak \GL(n,R/\rho)$ for a
maximal ideal $\rho$ such that
\begin{itemize}
\item the kernel $G_\rho$ of $\varphi_\rho$ on $G$ is
unipotent-by-abelian (UA) if $G$ is SF; \item $G_\rho$ is
torsion-free if $\mathrm{char} \, \F = 0$.
\end{itemize}
To be more explicit, let $\F =\P = \Q(\alpha)$ where $\alpha$ has
minimal polynomial $f(X)$, and let $G=\langle \mathcal{S}\rangle$.
Then $\varphi_\rho$ on $R\cap \Q$ is reduction modulo an odd prime
$p\in \Z$ not dividing the discriminant of $f(X)$ nor the
denominators of entries in elements of $\mathcal{S}\cup
\mathcal{S}^{-1}$. Hence $\varphi_\rho$ maps $R$ into the finite
field $\Z_p(\beta)$, where $\beta$ is a root of the mod $p$
reduction of $f(X)$. We adhere to this notation from \cite{Tits}.

\subsection{Unipotent groups}\label{UnipotGpsSect}

Denote the group $\UT(n,K)$ of upper unitriangular matrices over a
commutative unital ring $K$ by $T$. Define $T_i$ to be the
subgroup of $T$ consisting of all matrices with their first $i-1$
superdiagonals equal to zero. Then $T=T_1 > T_2 > \cdots >
T_{n}=1$ is the lower (and upper) central series of $T$. The
homomorphism on $T_i$ that maps each element to its $i$th
superdiagonal has kernel $T_{i+1}$ and image the 
$(n-i)$-fold direct sum $K^+ \! \oplus \cdots \oplus K^+$.
\begin{lemma}\label{GinUTQ}
If $G \leq \UT(n,\Q)$ then
\begin{itemize}
\item[{\rm (i)}] $G$ is polyrational, \item[{\rm (ii)}] $\rk(G) =
\hi(G) \leq n(n-1)/2$.
\end{itemize}
\end{lemma}
\begin{proof}
Let $K = \Q$ in the notation introduced just before the lemma.
Since $(G\cap T_i)/(G\cap T_{i+1})$ is isomorphic to a subgroup of
$T_i/T_{i+1}$, (i) is clear. Then $\rk(G) = \hi(G)$ by
Proposition~\ref{RkeqHPolyrat}. Also, by Lemma~\ref{RkHFormulae}
(ii),
\begin{align*}
\hi(T) &= \hi(T_1/T_2) + \hi(T_2/T_3) + \cdots +
\hi(T_{n-1}/T_n)\\
& = {\textstyle \sum}_{i=1}^{n-1} i
 = n(n-1)/2. \qedhere
\end{align*}
\end{proof}

\begin{corollary}\label{GinUTP}
If $G\leq \UT(n,\P)$ then $G$ is polyrational and $\rk(G) = \hi(G)
\leq \allowbreak nm(nm-1)/2$.
\end{corollary}

\subsection{Ranks of solvable-by-finite groups over number fields}
\label{SFSubgpsGLnP}

In this section $G$ is a finitely generated subgroup of
$\GL(n,\P)$. We prove that if $G$ is SF then it has finite rank.
Although $\rk(G)$ can be arbitrarily large, the ranks of finitely
generated SF subgroups of $\GL(n,\mathcal{O}_\P)$ are bounded by
functions of $n$ and $m$, which we give below.
%
%
\begin{proposition}\label{FgSFoverP}
Suppose that $G$ is SF. Then $G$ is polyrational-by-finite, hence
of finite Pr\"{u}fer rank.
\end{proposition}
\begin{proof}
Select an ideal $\rho$ such that $G_\rho$ is UA and $G/G_\rho$ is
finite. Let $U$ be the unipotent radical of $G_\rho$; then
$G_\rho/U$ is finitely generated abelian. Write $G_\rho/U = H/U
\times \tau(G_\rho/U)$. Since $H/U$  is a finitely generated free
abelian group and $U$ is conjugate to a subgroup of $\UT(n,\P)$,
$H$ is polyrational. Thus $G_\rho$ has a polyrational normal
subgroup of finite index. Consequently the same is true for $G$.
\end{proof}
\begin{remark}\label{RkHBreakdown}
Retaining the notation in the proof of
Proposition~\ref{FgSFoverP}, $\hi(G) = \hi(G_\rho)$ and $\rk(G)
\leq \allowbreak \rk(G_\rho) + \rk(\varphi_\rho(G))$ by
Lemma~\ref{RkHFormulae}. Furthermore $\rk(G_\rho) \leq \hi(H) +
\rk(\tau(G_\rho/U))$. If we know $x\in \allowbreak \GL(n,\P)$ that
conjugates $G$ to block upper triangular form with completely
reducible diagonal blocks, then we can choose $\rho$ so that the
torsion-free group $G_\rho$ is polyrational, and thus
$\rk(G_\rho)=\allowbreak \hi(G_\rho)$. In particular, $G_\rho$ is
polyrational for any $\rho$ when $G$ is completely reducible.
\end{remark}

Remark~\ref{RkHBreakdown} underpins our algorithm to calculate
ranks.
\begin{corollary}
A finitely generated subgroup of $\GL(n,\F)$ has finite Pr\"{u}fer
rank if and only if it is SF and $\Q$-linear.
\end{corollary}

\begin{proposition}\label{FgFRiffSF}
The following are equivalent.
\begin{itemize}
\item[{\rm (i)}] $G$ is SF. \item[{\rm (ii)}] $G$ has finite
Pr\"{u}fer rank. \item[{\rm (iii)}] $G$ has finite torsion-free
rank.
\end{itemize}
\end{proposition}
\begin{proof}
Theorem~\ref{GroundZero} and Proposition~\ref{FgSFoverP} give (i)
$\Leftrightarrow$ (ii). Then (i) $\Leftrightarrow$ (iii) by
Lemma~\ref{FinitePMeansFiniteH} and the Tits alternative.
\end{proof}
\begin{remark}
Thus, we can test whether $G$ has finite rank using the algorithm
of \cite[Section~3.2]{Tits}, which decides the Tits alternative
for $G$. This algorithm accepts a finitely generated linear group
over any $\F$; if it returns $\tt false$, then the input does not
have finite rank.
\end{remark}

In fact, Proposition~\ref{FgSFoverP} holds for a wider class of
groups: what is most important here is that unipotent subgroups of
$\GL(n,\P)$ have finite rank.
%
%
\begin{lemma}\label{FGTooStong}
If $R$ is a finitely generated subring of $\, \P$ then an SF
subgroup $H$ of $\GL(n,R)$ has finite Pr\"{u}fer rank.
\end{lemma}
\begin{proof}
It suffices to confirm that $H/U(H)$ has finite rank. Indeed,
$H/U(H)$ is finitely generated AF by \cite[4.10, \mbox{p.
57}]{Wehrfritz}.
%
%
%
\end{proof}


\begin{proposition}\label{RkBounds}
Suppose that $G\leq \GL(n,\mathcal{O}_\P)$ is SF. Then $\hi(G)
\leq \allowbreak nm(nm+1)/2$ and $\rk(G)\leq \allowbreak
nm(2nm+3)/2$.
\end{proposition}
\begin{proof}
Since $\GL(n, \mathcal{O}_\P)$ embeds into $\mathrm{GL}(nm,\Z)$,
we may assume without loss of generality that $G\leq \GL(n,\Z)$.

(i) \phantom{} Suppose that $G$ is abelian and $\Q$-irreducible.
Then the enveloping algebra $\langle G\rangle_\Q$ is a number
field of degree $n$ over $\Q$. Moreover, $G$ is contained in the
unit group of the ring of integers of $\langle G\rangle_\Q$. Hence
$\rk(G)\leq n$ by Dirichlet's Units Theorem \cite[Theorem 12.6,
\mbox{p.}~227]{StewTall}.

(ii) \phantom{} If $G$ is abelian and completely reducible over
$\Q$, then \cite[Lemma 4, \mbox{p.}~173]{Supr} implies that $G$ is
conjugate to a group of block diagonal matrices $\{
\mathrm{diag}(\mu_1(g), \ldots , \mu_k(g)) \ | \ g \in G\}$ where
$\mu_i(G)\leq \allowbreak \GL(n_i,\Z)$ is $\Q$-irreducible.
Therefore, by (i),
\[
\rk(G)\leq {\textstyle \sum}_{i=1}^k\rk(\mu_i(G)) = {\textstyle
\sum}_{i=1}^kn_i = n.
\]

(iii) \phantom{} If $G$ is UA then $\rk(G) \leq \frac{n(n-1)}{2} +
n = n(n+1)/2$ by (ii) and Lemma~\ref{GinUTQ}.

(iv) \phantom{} By Remark~\ref{RkHBreakdown}, there is an odd
prime $p$ such that $\hi(G) = \rk(G_\rho)$ and $\rk(G)
\leq\allowbreak \rk(G_\rho) +\allowbreak \rk(\varphi_\rho(G))$ for
$\rho = pR$. Thus $\hi(G) \leq n(n+1)/2$. 
By \cite{KovacsRobinson}, a finite completely reducible linear
group of degree $n$ can be generated by $\lfloor 3n/2\rfloor$
elements. Since $\rk(\UT(n,p)) \leq\allowbreak n(n-1)/2$, we
deduce that $\rk(\varphi_\rho(G))\leq n(n+2)/2$. The stated bound
on $\rk(G)$ follows.
\end{proof}
\begin{remark}\label{InPFCase}\
(i) If $n\geq 4$ then the bound on $\rk(G)$ in
Proposition~\ref{RkBounds} can be improved using $\rk(\GL(n,p))
\leq \allowbreak \frac{n^2}{4} +1$; see \cite[\mbox{p.}
199]{Pyber}.

(ii) $\rk(\GL(n,p)) \geq \lfloor n^2/4\rfloor$ because $\UT(n,p)$
has an elementary abelian subgroup of order $p^{\lfloor
n^2/4\rfloor}$.
%
\end{remark}

\section{Subgroups of finite index} \label{Sect3}

In this section we first derive a rank-based criterion to
recognize when a subgroup of a finitely generated linear group of
finite rank has finite index. Subsequently we prove a result about
the unipotent radical that forms a key piece of our main
algorithm.

\subsection{Ranks and isolators}\label{RandISection}

We recall some definitions from \cite[\mbox{pp.}~83--86]{LenRob}.
The \emph{$p$-rank} ($p$ prime) of an abelian group is the
cardinality of a maximal $\Z_p$-linearly independent subset of
elements of order $p$. A solvable group $G$ has \emph{finite
abelian ranks} ($G$ is a \emph{solvable FAR group}) if there is a
series of finite length in $G$ with each factor abelian, and of
finite torsion-free rank and finite $p$-rank for every prime $p$.
%
%
A \emph{minimax group} is a group that has a series of finite
length whose factors satisfy either the maximal condition or the
minimal condition on subgroups. The minimality $m(G)$ of a
solvable minimax group $G$ is the number of infinite factors in a
series of $G$ with each factor finite, cyclic, or quasicyclic. For
finitely generated solvable groups, the notions of FAR, minimax,
and finite Pr\"{u}fer rank all coincide
\cite[\mbox{pp.}~175--176]{LenRob}.

The following theorem and its proof were communicated to us by
D.~J.~S.~Robinson.
\begin{theorem}[D.~J.~S. Robinson] \label{SolvableFAR}
Let $H$ be a subgroup of a finitely generated solvable FAR group
$G$. Then $|G:H|$ is finite if and only if $\hspace*{.35mm} \hi(H)
= \hi(G)$.
\end{theorem}
\begin{proof}
The `only if' direction being clear, assume that $\hi(H) =
\hi(G)$. For $N \unlhd G$,
\begin{align*}
\hi(HN/N) &= \hi(H) - \hi(H \cap N) \\
& \geq \hi(G) - \hi(N) = \hi(G/N).
\end{align*}
Thus $\hi(HN/N) = \hi(G/N)$. We prove that $|G : H|$ is finite by
induction on $m(G)$. If $m(G)=0$ then $G$ is finite, so let
$m(G)>0$.

Denote the finite residual
of $G$ by $D$; this
is a divisible periodic abelian group
\cite[5.3.1, \mbox{p.}~96]{LenRob}.
Suppose that $D \neq 1$. Then $m(G/D) < m(G)$,
and by the inductive hypothesis
$|G : HD|$ is finite. Hence $HD$ is finitely generated, so $HD =
HD_0$ where $D_0 \leq D$ is finitely generated, i.e., finite.
This implies that $|HD : H|$ is finite, as is $|G : H|$.

Suppose now that $D = 1$. Then $G$ has a non-trivial torsion-free
abelian normal subgroup $A$ (for example, the penultimate term in
the derived series of a non-trivial torsion-free normal subgroup
of $G$).
Since $m(G/A) < m(G)$,
by induction $|G : HA|$ is finite.
Next, $H \cap A \neq 1$; otherwise $\hi(H) = \allowbreak \hi(HA/A)
= \hi(G/A) < \hi(G)$. So
the result holds for $HA/(H \cap A)$ and its subgroup $H/(H\cap
A)$ by induction.
Therefore $|HA : H|$ is finite, as is $|G : H|$.
\end{proof}


\begin{remark}
Finitely generated linear groups are residually finite \cite[4.2,
\mbox{p.}~51]{Wehrfritz}, so for our algorithms we only need that
part of the proof of Theorem~\ref{SolvableFAR} in which $D=1$.
\end{remark}

\begin{corollary}\label{RobKey}
Let $H\leq G\leq \GL(n,\F)$ where $G$ is finitely generated and of
finite Pr\"{u}fer rank. Then $|G:H|$ is finite if and only if
$\hspace*{.35mm} \hi(H) = \hi(G)$.
\end{corollary}


The \emph{isolator} in $G$ of a subgroup $H$ is
\[
I_G(H) = \{ x \in G \mid x^k \in H \text{ for some positive
integer } \, k\} .
\]
\begin{theorem}\label{LenRobKey}
Let $G$ be a finitely generated SF group, and let $H\leq G$. Then
$|G:H|$ is finite if and only if $I_G(H) = G$.
\end{theorem}
\begin{proof}
See \cite[2.3.14, \mbox{p.}~45]{LenRob}.
\end{proof}

\begin{lemma}\label{IsolatorAll}
Suppose that $G$ is a solvable FAR group with a finitely generated
subgroup $H$ such that $\hi(H) =\allowbreak \hi(G)$. Then
$I_G(H)=G$.
\end{lemma}
\begin{proof}
Since $\hi(\langle g , H\rangle) = \hi(H)$ for every $g\in G$,
the lemma follows from Theorem~\ref{SolvableFAR}.
\end{proof}

\begin{lemma}\label{IsolatorImpliesRk}
Suppose that $G$ is a group of finite torsion-free rank, and $H$
is a subgroup of $G$ such that $I_G(H) = G$. Then $\hi(G)=\hi(H)$.
\end{lemma}

We consider an illustrative example. Let $G\leq \UT(n,\C)$ be an
algebraic group defined over $\Q$, and set $G_S := G\cap \GL(n,S)$
for a subring $S$ of $\C$.
%
%
Recall that $L\leq G_\Q$ is an arithmetic subgroup of $G$ if $L$
is commensurable with $G_\Z$; i.e., $L\cap G_\Z$ has finite index
in both $L$ and $G_\Z$.
\begin{lemma}\label{UnipotentArithmeticity}
A finitely generated subgroup $L$ of $G_\Q$ is an arithmetic
subgroup of $G$ if and only if $\, \rk(L) = \rk(G_\Q)$.
\end{lemma}
\begin{proof}
By \cite[Lemma 6, \mbox{p.}~138]{Segal}, $H:= L\cap G_\Z$ has
finite index in $L$. Since $L$ is polyrational and nilpotent,
$\rk(H) = \rk(L)$ by Theorem~\ref{SolvableFAR}. Similarly (as
$G_\Z$ is finitely generated) $|G_\Z:H|< \allowbreak \infty$ if
and only if $\rk(G_\Z) = \rk(H)$. Also, it is not difficult to
verify that $G_\Q = I_{G_\Q}(G_\Z)$. Hence $\rk(G_\Q)=\rk(G_\Z)$
by Lemma~\ref{IsolatorImpliesRk}.
\end{proof}
\begin{remark}
By Lemma~\ref{UnipotentArithmeticity} and \cite[Corollary
7.2]{deGraaf}, if $L$ is arithmetic in $G$ then $\hi(L)$ is the
dimension of $G$ as an algebraic group.
\end{remark}
%
%

\subsection{Pr\"{u}fer rank of a unipotent normal subgroup}

Let $G$ be a finitely generated SF subgroup of $\GL(n,\P)$.
%
%
We show how to construct a finitely generated subgroup of $U(G)$
with the same Pr\"{u}fer rank as $U(G)$.


Suppose that $G = \gp{x_1, \ldots , x_r}$, and let $Y$ be a finite
subset of $U(G)$. The normal closure $N=\allowbreak \gp{Y}^G$ is
in $U(G)$. Define subgroups $H_1\leq H_2 \leq \cdots$ of $N$ as
follows: let $H_1 = \gp{Y}$, and for $i\geq \allowbreak 1$, if
$H_i = \gp{y_{i1},\ldots , y_{is_i}}$ then
\[
H_{i+1} = \gp{y_{ij}, y_{ij}^{x_k}, y_{ij}^{x_{k}^{-1}} : 1\leq j
\leq s_i, \, 1\leq k \leq r}.
\]
Since $\rk(H_i) \leq \rk(H_{i+1}) \leq \rk(N)$, there exists $t$
such that $\rk(H_t) = \rk(H_{t+1})$.

\begin{lemma}\label{RkUGfromNormalGens}
$\rk(H_t) = \rk(N)$.
\end{lemma}
\begin{proof}
By Lemma~\ref{IsolatorAll}, $I_{H_{t+1}}(H_t)=H_{t+1}$. So for
$1\leq i \leq r$ and $1\leq j \leq s_t$, there are positive
integers $m_{ij}, \bar{m}_{ij}$ such that
$(y_{tj}^{x_i})^{m_{ij}}$, $(y_{tj}^{x_i^{-1}})^{\bar{m}_{ij}} \in
H_t$. We claim that $y_{tj}^x\in I_G(H_t)$ for all $j$ and $x\in
\allowbreak G$. First,
\[
(y_{tj}^{x_vx_u^{\pm 1}})^{m_{vj}} =
((y_{tj}^{x_v})^{m_{vj}})^{x_u^{\pm 1}} \in H_{t+1}
\]
since $H_i^{x_k^{\pm 1}} \leq H_{i+1}$. Similarly
$(y_{tj}^{x_v^{-1}x_u^{\pm 1}})^{\bar{m}_{vj}}\in H_{t+1}$.
Induction on the word length of $x$ then
establishes that $y_{tj}^x\in I_G(H_t)$ as claimed. Hence $N
=H_1^G \leq H_t^G \subseteq I_G(H_t)$; i.e., $N =I_N(H_t)$. By
Lemma~\ref{IsolatorImpliesRk}, the proof is complete.
\end{proof}

\section{Computing ranks of solvable-by-finite
linear groups} \label{AlgSect}

Let $\mathcal{S}$ be a finite subset of $\GL(n,\P)$ where
$|\P:\Q|=m$, and let $G=\gp{\mathcal{S}}$. In this section we
present algorithms to compute $\hi(G)$ and a bound on $\rk(G)$.
These lead directly to an algorithm that tests whether a finitely
generated subgroup of $G$ has finite index.

Proposition~\ref{FgFRiffSF} allows us first to test whether $G$
has finite Pr\"{u}fer (and thereby torsion-free) rank: ${\tt
IsFiniteRank}(G)$ returns $\tt true$ precisely when the procedure
${\tt IsSolvableByFinite}(G)$ as in \cite[\mbox{p.}~402]{Tits}
returns $\tt true$. From now on, $G$ has finite rank.
%
%

\subsection{Auxiliary procedures}

\subsubsection{}\label{CRedAbCase}
Suppose that $G$ is abelian and irreducible. Methods to construct
a presentation of $G$ are reasonably standard;
see \cite[Chapter~4]{AssmannDiplom} for details. We can find the
homogeneous components of $G$ (e.g., by \cite{Ronyai}), so the
methods extend to completely reducible abelian $G$. For such input
we have procedures (i)~${\tt PresentationA}$, which returns a
presentation of $G$; and (ii)~${\tt RankA}$, which returns the
torsion-free rank of $G$. Then $\rk(G) = {\tt RankA}(G) +
\varepsilon$ where $\varepsilon = 0$ if $G$ is torsion-free and
$\varepsilon =\allowbreak 1 $ otherwise.

\subsubsection{}\label{UTCase}
If $G\leq \UT(n,\P)$ then $G$ is isomorphic to a subgroup of
$\UT(nm,\Z)$ \cite[Lemma~2, \mbox{p.}~111]{Segal}.
Since $\UT(nm,\Z)$ is polycyclic, a constructive polycyclic
sequence for $G$ may be calculated as in \cite[Chapter 9]{Sims}
or \cite[Chapter 5]{AssmannDiplom}. From this one immediately
reads off ${\tt RankU}(G):= \allowbreak \hi(G) =\allowbreak
\rk(G)$.



\subsection{Completely reducible groups}\label{CRedCase}

If $G$ is completely reducible then $G_\rho$ is completely
reducible abelian and $\hi(G) = \allowbreak \hi(G_\rho)$. Thus
${\tt RankCR}(G):=\hi(G) = {\tt RankA}(G_\rho)$ as per
\ref{CRedAbCase}.

Now let $\F$ be arbitrary and $G\leq \GL(n,\F)$ be finitely
generated SF. In \cite[Section~4]{Tits} we show how to test
whether $G$ is completely reducible. Here we describe a more
general procedure.

We refer to \cite[Section~3.2]{Tits}.
The computations carried out in a run of ${\tt
IsSolvableByFinite}(G)$ yield a change of basis matrix $x$ such
that $G^x$ is block upper triangular and all diagonal blocks of
$G_\rho^x$ are abelian. Treating each diagonal block of $G^x$
separately, assume that $G_\rho$ is abelian. Let $M= \allowbreak
\{ h_1, \ldots , h_t\} = \allowbreak {\tt
NormalGenerators}(G_\rho)$; i.e., $G_\rho = \langle M \rangle^G$.
With a subscript `$u$' denoting unipotent part from a Jordan
decomposition, $H = \langle (h_1)_u, \ldots ,
(h_t)_u\rangle=\allowbreak  \langle M \rangle_u\leq (G_\rho)_u$.
Set $U=\allowbreak \mathrm{Fix}((G_\rho)_u)$ and $W= \allowbreak
\mathrm{Fix}(H)$. Since $G$ normalizes $(G_\rho)_u$, we see that
$U$ is a $G$-module. We find $U$ as follows.

\vspace*{7.5pt}

\begin{enumerate}

\item $\bar{W}:= W$.

\item While $\exists$ $g_i \in \mathcal{S}$ such that
$g_i\bar{W}\neq \bar{W}$

    \ \ \ $\bar{W}:= g_i\bar{W}\cap \bar{W}$.

\item Return $\bar{W}$.

\end{enumerate}

\vspace*{10pt}

\noindent Clearly $U\subseteq \bar{W}$. Let $v\in \bar{W}$ and
$g\in G$; then $(h_i)_u^gv = g^{-1}(h_i)_u. gv = g^{-1}gv$
(because $gv\in \bar{W} \subseteq \allowbreak W)=v$. This shows
that $\bar{W} = U$. By \cite[Theorem~5, \mbox{p.} 172]{Supr}, $U$
is completely reducible as a $G_\rho$-module. Therefore, if
$\mathrm{char}\, \F$ does not divide $|G:G_\rho|$, then $U$ is a
completely reducible $G$-module by \cite[Theorem~1,
\mbox{p.}~122]{Supr}. Repeat the previous computation after
replacing the current underlying space $V$ for $G$ by $V/U$.
Continuing in this fashion, we eventually produce a flag
$V=V_1\supset V_2 \supset \allowbreak  \cdots \supset V_l\supset
\{ 0\}$ of $G$-modules with each quotient $V_i/V_{i+1}$ completely
reducible.

We adopt the following notation in our pseudocode. For a matrix
group $H$ in block upper triangular form, $\mu$ denotes the
projection of $H$ onto its block diagonal, and $\mu_i$ is the
projection onto its $i$th diagonal block. When all diagonal blocks
are completely reducible, $\ker \mu = U(H)$ and $\mu(H)$ is a
`completely reducible part' of $H$.

\vspace*{12.5pt}

${\tt CompletelyReduciblePart}(G)$

\vspace*{1mm}

Input: a finite subset $\mathcal{S}$ of $\GL(n,\F)$ such that
$\mathrm{char}\, \F$ does not divide $|G:G_\rho|$ and $G=\langle
\mathcal{S}\rangle$ is SF.

Output: a generating set for a completely reducible part of $G$.

\vspace*{1mm}

\begin{enumerate}

\item \label{CRPartFirstStep} Replace $G$ by $G^x$ in block upper
triangular form with $k$ diagonal blocks, where $\mu(G_\rho^x)$ is
abelian.

\item $M:= {\tt NormalGenerators} (G_\rho)$.

\item For $i = 1$ to $k$, determine $x_i$ such that
$\mu_i(G)^{x_i}$ is block upper triangular with completely
reducible diagonal blocks, by the recursive calculation of fixed
point spaces for $\langle \mu_i(M)\rangle_u$.

\item \label{CRPartLastStep} Return $\mu(\mathcal{S}^y)$ where
$y=x \cdot \mathrm{diag}(x_1, \ldots , x_k)$.
\end{enumerate}

\vspace*{7.5pt}

\begin{remark}
If $G$ is nilpotent-by-finite then we can take $k=1$, $\mu_1=
\mathrm{id}$, and omit Step (\ref{CRPartFirstStep}).
\end{remark}

We need one other procedure for completely reducible $G\leq
\GL(n,\P)$: ${\tt PresentationCR}(G)$ returns a presentation of
$G$. This combines a presentation of $\varphi_\rho(G)$, computed
using the machinery of \cite{CT}, with ${\tt
PresentationA}(G_\rho)$.


\subsection{The unipotent radical}\label{UnipotentRadical}

Our next procedure is based on Lemma~\ref{RkUGfromNormalGens} and
its proof.

\vspace*{12.5pt}

${\tt RankOfUnipotentRadical}(G)$

\vspace*{1mm}

Input: a finite subset $\mathcal{S}=\{g_1,\ldots, g_r\}$ of
$\GL(n,\P)$ such that $G = \langle \mathcal{S}\rangle$ is SF.

Output: $\hi(U(G))=\rk(U(G))$.

\vspace*{1mm}

\begin{enumerate}

\item $\tilde{G}:=\langle {\tt
CompletelyReduciblePart}(G)\rangle$.

\item \label{URStep2} Find $X:= {\tt NormalGenerators}(U(G))$ from
${\tt PresentationCR} (\tilde{G})$.

\item  \label{URPenultimateStep} While ${\tt RankU}(\gp{x,
x^{g_i}, x^{g_{i}^{-1}} : x \in X , 1\leq i \leq r}) > {\tt
RankU}(\langle X\rangle)$ do

\ \  $X:= \{x, x^{g_i}, x^{g_{i}^{-1}} : x \in X , 1\leq i \leq
r\}$.

\item Return ${\tt RankU}(\langle X\rangle)$.
\end{enumerate}

\vspace*{5.5pt}

\begin{remark}
The finitely generated subgroup $H=\langle X\rangle$ of $U(G)$
such that $\rk(H) = \allowbreak \rk(U(G))$ found at the end of
Step~(\ref{URPenultimateStep}) could be valuable in further
computations with $G$.
\end{remark}

\subsection{Algorithms for computing ranks, and
an application}\label{ComputRksAppl}

Guided by Remark~\ref{RkHBreakdown}, we assemble our constituent
procedures into the final algorithms.

\vspace*{12.5pt}

${\tt HirschNumber}(G)$

\vspace*{1mm}

Input: a finite subset $\mathcal{S}$ of $\GL(n,\P)$ such that $G =
\langle \mathcal{S}\rangle$ is SF.

Output: $\hi(G)$.


\begin{enumerate}
\item[] Return ${\tt RankCR}(\langle {\tt
CompletelyReduciblePart}(G)\rangle)+{\tt
RankUnipotentRadical}(G)$.

\end{enumerate}

\vspace*{10pt}

Then ${\tt RankBound}(G):={\tt HirschNumber}(G)+\rk(\GL(nm,3))$ is
an upper bound on the Pr\"{u}fer rank of $G$ (see
Remark~\ref{InPFCase}).

Corollary~\ref{RobKey} gives us the following.

\vspace*{12.5pt}

${\tt IsOfFiniteIndex}(G,H)$

\vspace*{1mm}

Input: finite subsets $\mathcal{S}_1$, $\mathcal{S}_2$ of
$\GL(n,\P)$ such that $G = \langle \mathcal{S}_1\rangle$ is SF and
$H= \allowbreak \langle \mathcal{S}_2\rangle \leq G$.

Output: ${\tt true}$ if $|G:H|$ is finite; ${\tt false}$
otherwise.


\begin{enumerate}
\item[] Return ${\tt true}$ if ${\tt HirschNumber}(G) = {\tt
HirschNumber}(H)$; else return ${\tt false}$.

\end{enumerate}

\vspace*{5pt}

\subsection{The implementation}\label{ExpResults}

We have implemented our algorithms as part of the {\sc Magma}
package {\sc Infinite} \cite{Infinite}. An algorithm of Biasse and
Fieker \cite{Fieker} 
is used to work with irreducible abelian groups over number
fields.

We report on several examples below (these will be available in a
future release of {\sc Infinite}). Our experiments were performed
on a 2GHz machine using {\sc Magma} V2.19-6. The test groups are
conjugated to ensure that generators are not sparse and matrix
entries are large. Each time has been averaged over three runs. As
observed in \cite{Tits, Recog}, the single most expensive task is
evaluating relators to obtain normal generators for the congruence
subgroup.

\begin{enumerate}

\item $G_{1}$ is an irreducible non-abelian subgroup of
$\GL(2,\Q(i))$, $i= \sqrt{-1}$, and $G_2 \leq \GL(5, \Q)$ is a
solvable group from the database of maximal finite rational matrix
groups~\cite{NP95}. Then $G_{3}= \allowbreak G_{1} \otimes G_2$ is
a $5$-generator AF completely reducible subgroup of
$\GL(10,\Q(i))$. We compute $\hi(G_3)=3$ in $10$s.

\item $G_4 \leq G_3\otimes \UT(3,\Z)$ is a $15$-generator,
nilpotent-by-finite (NF), reducible but not completely reducible
subgroup of $\GL(30,\Q(i))$. We compute $\hi(G_4)=6$ in $87$s.

\item $G_5 \leq H \otimes \ T$ where $T$ is an upper triangular
subgroup of $\GL(6,\Q)$ and $H =\mathrm{diag}(H_1, H_2)$; $H_1$,
$H_2$ are maximal finite rational matrix groups of degrees $4$,
$2$ respectively. The $8$-generator group $G_5$ is SF and not NF.
We compute $\hi(G_5)= 7$ in $1104$s, and establish that a random
$4$-generator subgroup has infinite index in $163$s.

\item Let $a \in \GL(6,\Q)$ be of the form $\mathrm{diag}(1, 2,
\ldots)$ and let $b = {\tiny \left( \!
\renewcommand{\arraycolsep}{.125cm} \begin{array}{cc} x& y \\
0 & u \end{array} \! \right)}$ where $x = \allowbreak {\tiny
\left( \!
\renewcommand{\arraycolsep}{.125cm} \begin{array}{cc} 1& 1 \\
0 & 1 \end{array} \! \right)}$, $y$ is a non-zero $2\times 4$
matrix over $\Q$, and $u\in \UT(4, \Z)$. Then $G_{6}\leq
\allowbreak \GL(6, \Q(\sqrt{5}\, ))$ is conjugate to a group
generated by $a$, $b$, another diagonal matrix and two other
unipotent matrices in $\GL(6,\Q)$. Note that $G_6$ is SF but not
PF. We compute $\hi(G_6)=12$ in $18$s.

\item For each of $G_3$, $G_4$, $G_6$ we select random finitely
generated non-cyclic subgroups $\hat{G}_j$. To establish that
$\hat{G}_j$ has finite index in $G_j$ takes $4$s, $53$s, and $17$s
respectively.

\end{enumerate}



\subsubsection*{Acknowledgments}
Detinko and Flannery were supported by Science Foundation Ireland
grants 08/\allowbreak RFP/\allowbreak MTH1331 and
11/RFT.1/\allowbreak MTH/\allowbreak 3212-STTF11. O'Brien was
supported by the Marsden Fund of New Zealand grant UOA 1015. We
are very grateful to D.~J.~S.~Robinson for
allowing us to include his proof of Theorem~\ref{SolvableFAR} in
the paper.

\bibliographystyle{amsplain}

\end{document}